\documentclass[a4paper,11pt]{amsart}
\usepackage{latexsym,bm}
\usepackage{mathrsfs,amsmath,amssymb,amsthm,enumerate}
\usepackage[color,arrow,matrix]{xy}
\usepackage{color}
\usepackage{tikz}
\usepackage{graphicx}
\usepackage{pdfsync}

\theoremstyle{plain}
\newtheorem{theorem}{Theorem}[section]
\newtheorem{proposition}[theorem]{Proposition}
\newtheorem{lemma}[theorem]{Lemma}
\newtheorem{corollary}[theorem]{Corollary}
\numberwithin{equation}{section}

\theoremstyle{definition}
\newtheorem{definition}[theorem]{Definition}
\newtheorem{remark}[theorem]{Remark}
\newtheorem{example}[theorem]{Example}

\newcommand{\R}{\mathbb{R}}
\newcommand{\Z}{\mathbb{Z}}

\newcommand{\CP}{\mathbb{C}P}

\newcommand{\cE}{\mathcal{E}}

\renewcommand{\a}{\mathbf{a}}
\renewcommand{\b}{\mathbf{b}}
\renewcommand{\c}{\mathbf{c}}

\renewcommand{\v}{\mathbf{v}}
\newcommand{\w}{\mathbf{w}}
\renewcommand{\1}{\mathbf{1}}
\newcommand{\balpha}{\boldsymbol{\alpha}}
\newcommand{\bbeta}{\boldsymbol{\beta}}

\DeclareMathOperator{\wed}{wed}
\DeclareMathOperator{\link}{Lk}

\DeclareMathOperator{\proj}{Proj}

\DeclareMathOperator{\supp}{supp}
\DeclareMathOperator{\inv}{inv}

    \tikzstyle{v}=[circle, draw, solid, fill, inner sep=0pt, minimum width=4pt]

\begin{document}
\title[Small covers over wedges of polygons]{Small covers over wedges of polygons}

\author[S.Choi]{Suyoung Choi}
\address{Department of Mathematics, Ajou University, 206, World cup-ro, Yeongtong-gu, Suwon, 443-749, Republic of Korea}
\email{schoi@ajou.ac.kr}

\author[H.Park]{Hanchul Park}
\address{School of Mathematics, Korea Institute for Advanced Study (KIAS), 85 Hoegiro Dongdaemun-gu, Seoul 130-722, Republic of Korea}
\email{hpark@kias.re.kr}

\thanks{The first author was supported by Basic Science Research Program through the National Research Foundation of Korea(NRF) funded by the Ministry of Science, ICT \& Future Planning(NRF-2016R1D1A1A09917654).}

\date{\today}



\begin{abstract}
    A small cover is a closed smooth manifold of dimension $n$ having a locally standard $\Z_2^n$-action whose orbit space is isomorphic to a simple polytope. A typical example of small covers is a real projective toric manifold (or, simply, a real toric manifold), that is, a real locus of projective toric manifold. In the paper, we classify small covers and real toric manifolds whose orbit space is isomorphic to the dual of the simplicial complex obtainable by a sequence of wedgings from a polygon, using a systematic combinatorial method finding toric spaces called puzzles.
\end{abstract}

\maketitle

\tableofcontents
\section{Introduction}
    One of basic objects in toric topology is the \emph{small cover} which is an $n$-dimensional closed smooth manifold with a locally standard $\Z_2^n$-action whose orbit space is a simple convex polytope.
    It was firstly introduced in \cite{Davis-Januszkiewicz1991} as a generalization of real projective toric variety; when $X$ is a toric variety of complex dimension $n$, there is a canonical involution on $X$ and its fixed points set forms a real subvariety $X^\R$ of real dimension $n$, called a \emph{real toric variety}. It should be noted that $X^\R$ admits a $\Z_2^n$-action induced from the action of torus $T^n = (S^1)^n$ on $X$, and its orbit space is equal to $X/T^n$.
    A non-singular complete toric variety $X$ is called a \emph{toric manifold}, and the corresponding real toric variety $X^\R$ is called a \emph{real toric manifold}. Hence, if a toric manifold $X$ is projective, then its orbit space is a convex polytope, and hence $X^\R$ is a small cover.
    For an $n$-dimensional small cover $M$, the boundary complex $K$ of its orbit space $M/{\Z_2^n}=:P$ is a polytopal simplicial complex of dimension $n-1$. We simply say that $M$ is a small cover over $K$ or a small cover over $P$.

    Two small covers are said to be \emph{Davis-Januskiewicz equivalent} (simply, \emph{D-J equivalent}) if there is a weakly equivariant homeomorphism between them which covers the identity map of their orbit spaces.
    It is known by \cite{Davis-Januszkiewicz1991} that small covers are classified by a pair of a polytopal simplicial complex $K$ and a $\Z_2$-characteristic map $\lambda$ over $K$.
    The definition of $\Z_2$-characteristic map will be given in the next section.
    One fundamental problem in toric topology is to classify small covers as well as (real) toric manifolds up to D-J equivalence. See \cite{Garrison-Scott2003, Cai-Cehn-Lu2007, Choi2008, Wang-Chen2012, Choi-Park2016}.

    However, the class of small covers (up to D-J equivalence) is too huge to be classified.
    Hence, we have to restrict our attention to a smaller but interesting subclass of manifolds.
    We denote by $P_m$ the boundary complex of the regular $m$-gon, and by $(\partial I^n)^\ast$ the boundary complex of the $n$-dimensional cube.
    Two remarkable classes of polytopal simplicial complexes are ones obtainable by a sequence of wedges (definition will be given in Section 2) from either (i) $(\partial I^n)^\ast$ or (ii) $P_m$. A simplicial complex in the class (ii) is called a \emph{wedged polygon}, and it is denoted by $P_m(J)$ for some positive integer $m$-tuple $J \in \Z_+^m$.
    The reason why such classes are interesting came from the classical results for toric manifolds.
    Up to present, there are complete classifications of toric manifolds up to Picard number $3$ (\cite{Kleinschmidt1988,Batyrev1991,Choi-Park2016}).
    Interestingly, if $X$ is a toric manifold of Picard number $\leq 3$, then the boundary complex $K$ of its orbit space should be in the class (i) or (ii).
    In general, if the complex dimension of $X$ is $n$ and its Picard number is $m-n$, then $K$ is $n-1$ dimensional star-shaped complex having $m$ vertices. Therefore, if $m-n=2$, then $K$ must be the boundary complex of the product of two simplices \cite{Ziegler1995book} which can be obtained by a sequence of wedges from $P_4= (\partial I^2)^\ast$.
    Furthermore, it is shown in \cite{Gretenkort-Kleinschmidt1990} that if $m-n=3$ and $K$ supports a toric manifold, then $K$ is obtainable by a sequence of wedges from either $P_5$ or $(\partial I^3)^\ast$.
    Therefore, it is reasonable to focus on toric spaces whose orbit space has the boudary complex belonging in the class (i) or (ii).

    A study of toric manifolds and small covers over $K$ in the class (i) is well established.
    Note that a simplicial complex $(\partial I^n)^\ast(J)$ in the class (i) is the boundary complex of the product of simplices.
    A toric manifold over $(\partial I^n)^\ast(J)$ is known as a generalized Bott manifold, and its real part is called a generalized real Bott manifold.
    In \cite{Choi-Masuda-Suh2010}, it has been shown that every small cover over $(\partial I^n)^\ast(J)$ is indeed a generalized real Bott manifold and they have been classified up to D-J equivalence in terms of block matrices.

    Nevertheless, for the class (ii), a study of toric manifolds and small covers over $P_m(J)$ has not been established for long times except for a wedged pentagon $P_5(J)$ which is in the Picard number $3$ case.
    Kleinschmidt and Sturmfels \cite{Kleinschmidt-Sturmfels1991} showed that every toric manifold over $P_5(J)$ is projective, and, by using this fact, Batyrev \cite{Batyrev1991} classified toric manifolds over $P_5(J)$.

    Very recently, the authors in \cite{Choi-Park2016, CP_wedge_2} found a new way how to find all characteristic maps over a wedge of $K$ from information of $K$, and, using this new technique, in \cite{CP_CP_variety}, they have classified all toric manifolds over $P_m(J)$  in the class (ii) and showed that all toric manifolds over $P_m(J)$  are projective. This generalizes both the main results of \cite{Kleinschmidt-Sturmfels1991} and \cite{Batyrev1991}. Furthermore, they could count the number of the small covers (and real toric manifolds) over $P_5(J)$ in \cite{Choi-Park2016}.
    Hence, as the next step, it is natural to classify small covers or real toric manifolds over $P_m(J)$, as stated in Question~6.3 of \cite{CP_wedge_2}. This is our main purpose.

    This paper is organized as follows. In Section~\ref{sec:background}, {we review some notions including simplicial wedges, characteristic maps, and diagrams and puzzles. } In Section~\ref{sec:diagram_of_P_m}, we compute the diagram of $P_m$, and, by using this, we give the classification of the small covers over $P_m(J)$ in Section~\ref{sec:classification_over_P_m(J)}. In addition, we classify real toric manifolds over $P_m(J)$ in Section~\ref{sec:real_toric_manifold_over_P_m(J)}.
    In the final section, we shall provide a summary of the paper. There, we enclose all combinatorial concepts required to understand the results, and present the main results and examples. This section would be helpful for the readers even if s/he did not read other parts of the paper carefully.

\section{Backgrounds} \label{sec:background}
    The reader is recommended to refer \cite{CP_wedge_2} for many terms of this paper, even though we will present definitions of essential notions.

    Let $K$ be a simplicial complex on the vertex set $[m] := \{1,2,\dotsc,m\}$. We say that a subset $\tau \subseteq [m]$ of the vertex set $[m]$ is a \emph{minimal non-face} of $K$ if $\tau$ itself is not a face of $K$, but every proper subset of $\tau$ is a face of $K$. Every simplicial complex $K$ is determined by describing its minimal non-faces.

    \begin{definition}[\cite{BBCG2015}]
        Let $J = (j_1,\dotsc, j_m)$ be an $m$-tuple of positive integers. The simplicial complex $K(J)$ is defined by the vertex set
        \[
            \{ \underbrace{1_1,1_2,\ldots,1_{j_1}},\underbrace{{2_1},2_2,\ldots,{2_{j_2}}},\ldots, \underbrace{{m_1},\ldots,{m_{j_m}}} \}
        \]
        and the minimal non-faces
        \[
            \{ \underbrace{{(i_1)_1},\ldots,{(i_1)_{j_{i_1}}}},\underbrace{{(i_2)_1},\ldots,{(i_2)_{j_{i_2}}}},\ldots, \underbrace{{(i_k)_1},\ldots,{(i_k)_{j_{i_k}}}} \}
        \]
        for each minimal non-face $\{{i_1},\ldots,{i_k}\}$ of $K$.
    \end{definition}
    {In particular, for distinct vertices  $p_1,\dotsc,p_k \in [m]$ of $K$, denote by {$J_{p_1,\dotsc,p_k}$} the $m$-tuple such that the $p_i$th entries are 2, for $i=1,\dotsc,k$, and the other entries are 1, and, then }
    we use the notation
    \[
        \wed_{p_1,\dotsc,p_k}K {:= K(J_{p_1,\dotsc,p_k})}.
    \]
    When $k=1$, the operation $\wed_{p_1}K$ is known as the \emph{(simplicial) wedge} operation, and it is easy to show that $K(J)$ can be obtained from a sequence of wedges.
    See Figure~\ref{fig:wedge} for an example of the simplicial wedge.

    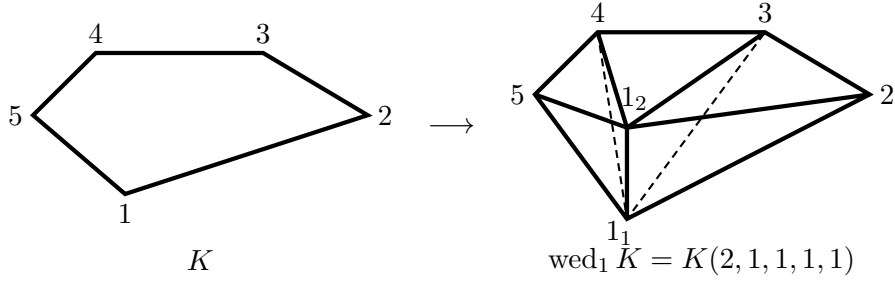
\begin{figure}[h]
        \begin{tikzpicture}[scale=.55]
            \coordinate [label=below:$1$](11) at (-9.8,0.6);
            \coordinate [label=right:$2$](22) at (-4,2.5);
            \coordinate [label=above:$3$](33) at (-6.5,4);
            \coordinate [label=above:$4$](44) at (-10.5,4);
            \coordinate [label=left:$5$](55) at (-12,2.5);
            \draw (-2,2.2) node {$\longrightarrow$};
            \coordinate [label={[xshift=-2.8pt,yshift=2.8pt]below:$1_1$}](1_1) at (2.2,0);
            \coordinate [label={[xshift=2.8pt,yshift=2.8pt]above:$1_2$}](1_2) at (2.2,2.2);
            \coordinate [label=right:$2$](2) at (8,3);
            \coordinate [label=above:$3$](3) at (5.5,4.5);
            \coordinate [label=above:$4$](4) at (1.5,4.5);
            \coordinate [label=left:$5$](5) at (0,3);
            \draw (-8,-1) node {$K$}
                (4,-1) node{$\wed_1K = K(2,1,1,1,1)$};
            \draw [ultra thick] (11)--(22)--(33)--(44)--(55)--cycle
                (1_1)--(2)--(3)--(4)--(5)--cycle
                (1_2)--(1_1)
                (1_2)--(2)
                (1_2)--(3)
                (1_2)--(4)
                (1_2)--(5);
            \draw [thick, densely dashed] (4)--(1_1)--(3);
        \end{tikzpicture}
        \caption{Illustration of a wedge of $K$}\label{fig:wedge}
    \end{figure}

    Let us further assume that $K$ is a polytopal simplicial complex, that is, $K$ is a simplicial complex isomorphic to the boundary of some simplicial polytope. Note that if $K$ is polytopal, then every $K(J)$ is also polytopal (\cite[Proposition~2.2]{Choi-Park2016}).
    \begin{definition}
        A \emph{$\Z_2$-characteristic map} over $K$, or simply a \emph{characteristic map}, is a map $\lambda \colon [m] \to \Z_2^m$ such that the following holds:
        \begin{equation}\label{eq:nonsingular}
            \text{if $\{i_1,\dotsc,i_n\}\in K$, then $\{\lambda(i_1),\dotsc,\lambda(i_n)\}$ is linearly independent.}\tag{$\ast$}
        \end{equation}
        The condition \eqref{eq:nonsingular} is known as the \emph{non-singularity condition}. Two characteristic maps $\lambda_1,\,\lambda_2 \colon [m] \to \Z_2^n$ over $K$ are called \emph{Davis-Januszkiewicz equivalent}, or \emph{D-J equivalent}, if there exists a linear isomorphism $\phi \in GL(n, \Z_2)$ such that $\lambda_2 = \phi \circ \lambda_1$. The equivalence classes themselves are called \emph{Davis-Januszkiewicz classes} or \emph{D-J classes}. 
    \end{definition}
    When $K$ is a polytopal simplicial complex of dimension $n-1$, it is a fundamental fact shown in \cite{Davis-Januszkiewicz1991} that $\Z_2$-characteristic maps over $K$ up to D-J equivalence correspond one-to-one to small covers over $K$ up to weakly $\Z_2^n$-equivariant homeomorphism fixing orbit spaces. A characteristic map $\lambda\colon [m] \to \Z_2^n$ is often represented by an $(n \times m)$-matrix called the \emph{characteristic matrix}, whose $i$th column vector corresponds to $\lambda(i)$ for $1\le i \le m$.
    Two characteristic matrices $\lambda_1$ and $\lambda_2$ correspond to two D-J equivalent characteristic maps if and only if $\lambda_2$ can be obtained from $\lambda_1$ by a series of elementary row operations.

    For a polytopal simplicial complex $K$ and a face $\sigma$ of $K$, its link $\link_K \sigma$ is also polytopal. If $K$ supports a characteristic map $\lambda$, then the link also supports a natural characteristic map over a link of $K$ called the \emph{projection} as follows.
    \begin{definition}
        For a characteristic map $\lambda \colon [m] \to \Z_2^n$ and a face $\sigma$ of $K$, the \emph{projection of $\lambda$ with respect to $\sigma$} is defined as
        \[
            (\proj_\sigma \lambda )(w) = [\lambda(w)]\in \Z_2^n/\langle \lambda(v)\mid v\in\sigma\rangle \cong \Z_2^{n - |\sigma|}
        \]
        when $w$ is a vertex of $\link_K \sigma$. The projection is well defined up to D-J equivalence. When $\sigma  = \{v\}$ is a vertex, one can simply write $\proj_\sigma\lambda = \proj_v \lambda$.
    \end{definition}
    \begin{remark}
    Note that a link of $K$ corresponds to a face of its dual simple polytope. Furthermore, the projection is the characteristic map for a submanifold fixed by a subtorus. When $\sigma$ is a vertex, then the submanifold is known as a \emph{characteristic submanifold}.
    \end{remark}
    \begin{remark}
        Throughout this paper, we will not necessarily distinguish the concept of characteristic maps and D-J classes unless otherwise mentioned. In actual calculation, we will mainly deal with characteristic maps while freely using elementary row operations. Also observe that the projection is actually an operation of D-J classes.
    \end{remark}

    Let $P_m$ be the boundary complex of the regular $m$-gon for $m\ge 3$.
    We call $P_m(J)$ a \emph{wedged polygon}. Our objective of the current paper is to figure out the D-J classes over $P_m(J)$. For this objective, let us review the definition of the diagram $D(K)$ of a polytopal simplicial complex $K$. Once $D(K)$ is calculated, one can construct every characteristic map over $K(J)$ for any $J$ due to \cite{CP_wedge_2}.

    \begin{definition}
    Let $K$ be a polytopal simplicial complex and define $V,E$, and $S$ as follows.
    \begin{enumerate}
        \item $V$ is the set of the D-J classes over $K$.
        \item $E$ is the set of the D-J classes over $\wed_pK$ where $p$ is a vertex of $K$.
        \item $S$ is the set of the D-J classes over $\wed_{p,q}K$ where $p$ and $q$ are distinct vertices of $K$.
    \end{enumerate}
    The triple $(V,E,S)$ is called the \emph{diagram} of $K$ and is denoted by $D(K)$. The elements of $V,E$, and $S$ are called the \emph{nodes}, \emph{edges}, and \emph{realizable squares} of $D(K)$ respectively.
    \end{definition}


    To any edge $\Lambda$ of $D(K)$, one can correspond the set $\{\lambda_1,\lambda_2,p\}$ where $\lambda_1, \lambda_2 \in V$ are two D-J classes over $K$ which are the two projections of $\Lambda$, namely, $\lambda_1 = \proj_{p_2}\Lambda$ and $\lambda_2 = \proj_{p_1}\Lambda$. This can be regarded as an edge connecting $\lambda_1$ and $ \lambda_2$ which is colored $p$, which looks like
     \[
        \xymatrix{
            \lambda_1 \ar@{-}[rr] ^p & & \lambda_2
        }.
    \]
    In this case, we say that $\lambda_1$ and $\lambda_2$ are \emph{$p$-adjacent}. We can abuse the set $\{\lambda_1, \lambda_2, p\}$ for the edge $\Lambda$ by the uniqueness result of \cite[Theorem~1.1]{Choi-Park2016}. That is, if both of the characteristic maps $\Lambda$ and $\Lambda'$ over $\wed_{p}K$ correspond to $\{\lambda_1, \lambda_2, p\}$, then either they are D-J equivalent or they are the ``twins'' of each other (see Remark~\ref{rem:twin}). An edge $\{\lambda_1,\lambda_2,p\}$ is said to be \emph{trivial} if $\lambda_1 = \lambda_2$. The edge set $E$ contains every possible trivial edge. That is, for every D-J class $\lambda \in V$ and every vertex $p \in K$, $\{\lambda,p\} \in E$.

    Like edges, for any realizable square $\Lambda$, we have the four projections $\lambda_1,\lambda_2,\lambda_3$, and $\lambda_4$ onto $K$, so that
    \begin{align*}
        \lambda_1 &= \proj_{\{p_2,q_2\}}\Lambda, \\
        \lambda_2 &= \proj_{\{p_1,q_2\}}\Lambda, \\
        \lambda_3 &= \proj_{\{p_2,q_1\}}\Lambda, \hbox{ and} \\
        \lambda_4 &= \proj_{\{p_1,q_1\}}\Lambda, \\
    \end{align*}
    and the four edges
    \begin{equation}\label{eq:square}
        \xymatrix{
                \lambda_{1} \ar@{-}[rr]^p \ar@{-}[dd]^q  & & \lambda_{2} \ar@{-}[dd]^q \\
                  &   &  \\
                \lambda_{3}\ar@{-}[rr]^p &   & \lambda_{4} }.
    \end{equation}
    We call a figure of the form \eqref{eq:square} a \emph{square}.

    \begin{remark}\label{rem:twin}
        For a D-J class $\Lambda$ over $\wed_pK$ such that $\lambda_1 = \proj_{p_2}\Lambda$ and $\lambda_2 = \proj_{p_1}\Lambda$, there surely exists a D-J class $\Xi$ over $\wed_pK$ such that $\lambda_1 = \proj_{p_1}\Xi$ and $\lambda_2 = \proj_{p_2}\Xi$, because of symmetry of the simplicial wedge. That means, every edge of $D(J)$ has its ``twin''. The twins must be distinguished when counting D-J classes over $\wed_pK$, but in most other situations they need not be distinguished. A similar argument goes for four different realizable squares over $\wed_{p,q}K$ corresponding to the same square up to symmetry.
    \end{remark}

    Here, we introduce a general method how to find D-J classes over $K(J)$ using the given diagram $D(K)$ referring to \cite{CP_wedge_2}. Let $K$ be a polytopal simplicial complex. Let $J=(j_1, \ldots, j_m) \in \Z_+^m$ be an {$m$-tuple of positive integers}.
    We consider an edge-colored graph $G(J)$ with $m$ colors constructed as follows: $G = G(J)$ is the graph determined by the $1$-skeleton of the simple polytope $\Delta^{j_1-1} \times \Delta^{j_2-1} \times \cdots \times \Delta^{j_m-1}$, where $\Delta^j$ is the $j$-dimensional simplex.
    One remarks that each edge $e$ of $G$ can be uniquely written as
    \[
        {e=} p_1 \times p_2 \times \dotsb \times p_{v-1}\times e_v \times p_{v+1} \times \dotsb \times p_m,
    \]
    where $p_i$ is a vertex of $\Delta^{j_i-1}$, $1\le i \le m$, $i\ne v$, and $e_v$ is an edge of $\Delta^{j_v-1}$.
    {Then} we color $v\in[m]$ on the edge $e$. Let us call a subgraph of $G(J)$ a \emph{subsquare} if it comes from a 2-face of $\Delta^{j_1-1} \times \Delta^{j_2-1} \times \cdots \times \Delta^{j_m-1}$ {which has $4$ edges}.
    Two edges
    \begin{align*}
       e &= p_1 \times \dotsb \times p_{v-1}\times e_v \times p_{v+1} \times \dotsb \times p_m \quad \text{ and } \\
       e' &= p'_1 \times \dotsb \times p'_{v-1}\times e'_v \times p'_{v+1} \times \dotsb \times p'_m
    \end{align*}of $G(J)$ are said to be \emph{parallel} if $e_v = e'_v$. Every parallel edge of $G(J)$ has the same color.

    Observe that the diagram $D(K)$ whose realizable squares are forgotten can be regarded as a pseudograph whose edges are colored.
    \begin{definition}
        A \emph{realizable puzzle} over $K$ is a pseudograph homomorphism $\pi\colon G(J) \to D(K)$ such that
        \begin{enumerate}
            \item $\pi$ preserves the coloring of the edges, and
            \item each image of the subsquare of $G(J)$ is a realizable square in $D(K)$.
        \end{enumerate}
    \end{definition}
    \begin{theorem}[\cite{CP_wedge_2}, Theorem~5.4]\label{thm:realizablepuzzle}
        There is a one-to-one correspondence
        \[
            \{\text{D-J classes over $K(J)$}\} \longleftrightarrow \{\text{realizable puzzles $ G(J) \to D(K)$}\}.
        \]
    \end{theorem}
    We have one more useful statement as follows.
    \begin{proposition}[\cite{CP_wedge_2}, Proposition~4.3]\label{prop:standardform}
        Let $v$ be any fixed node of $G(J)$ and suppose that there are two realizable puzzles $\pi,\,\pi'\colon G(J) \to D(K)$. Then $\pi = \pi'$ if and only if $\pi(v) = \pi'(v)$ and $\pi(w) = \pi'(w)$ for every node $w$ of $G(J)$ which is adjacent to $v$.
    \end{proposition}
    The essential form of the above statement is for the realizable squares. That is, if a square of the form \eqref{eq:square} is realizable, then $\lambda_4$ is uniquely determined by $\lambda_1$, $\lambda_2$, and $\lambda_3$. If this holds, we say that the two edges $\{\lambda_1,\lambda_2,p\}$ and $\{\lambda_1,\lambda_3,q\}$ \emph{spans a realizable square}.

    According to Remark~5.7 of \cite{CP_wedge_2}, for any edge $\{\lambda_1,\lambda_2,p\}$ and any vertex $q \neq p$, the square of the form
    
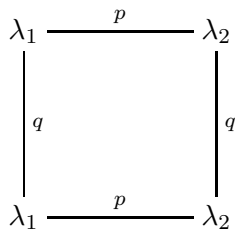
\[
        \xymatrix{
                \lambda_{1} \ar@{-}[rr]^p \ar@{-}[dd]^q  & & \lambda_{2} \ar@{-}[dd]^q \\
                  &   &  \\
                \lambda_{1}\ar@{-}[rr]^p &   & \lambda_{2} }
    \]
    is always realizable. We say the realizable square is \emph{reducible}. Realizable squares which are not reducible are called \emph{irreducible}.

    \section{The diagram $D(P_m)$} \label{sec:diagram_of_P_m}
    According to Theorem~\ref{thm:realizablepuzzle}, the diagram $D(K)$ can be used to construct every D-J class over $K(J)$. Therefore, our objective here is to provide the diagram $D(P_m)$. Throughout the paper, we identify the vertex set of $P_m$ with $[m]$ in counterclockwise order.
    \subsection{Node of $D(P_m)$}
    We begin by calculating the node set $V$ of $D(P_m)$.
    Let us label the vectors $\binom10,\binom01,\binom11 \in \Z_2^2$ by  $\a,\b,$ and $\c$, respectively. Then one can regard a characteristic map $\lambda\colon [m]\to \Z_2^2$ as a finite circular sequence consisting of $\a,\b$ and $\c$, no letter in which appears consecutively. Furthermore, any cardinality two subset of $\{\a,\b,\c\}$ is a basis of $\Z_2^2$ and the other element is the sum of the two, and hence any permutation of $\a , \b,$ and $\c$ does not affect the D-J type of the characteristic map $\lambda$.

    \begin{proposition}
        Let $t_m$ be the number of D-J classes over $P_m$. Then,
        $$t_m = \frac{2^{m-1} - (-1)^{m-1}}{3}.$$
    \end{proposition}
    \begin{proof}
      We may assume that $\lambda(1) = \a$ and $\lambda(2)=\b$. If $\lambda(m-1)=\a$, then the number of possible cases for $\lambda(3),\ldots, \lambda(m-2)$ is equal to $t_{m-2}$ and $\lambda(m)$ is either $\b$ or $\c$. So total number is $2t_{m-2}$. If $\lambda(m-1) \neq \a$, then $\lambda(m)$ is determined uniquely, and the number of possible cases for $\lambda(3),\ldots, \lambda(m-1)$ is equal to $t_{m-1}$. Therefore, $t_m = t_{m-1} + 2 t_{m-2}$, where $t_3= 1$ and $t_4=3$.
    \end{proof}

    It should be observed that a D-J class over $P_m$ can be regarded as a special kind of partitions of $[m]$: for each $\lambda$ (up to D-J equivalence), $[m]$ is divided into at most three disjoint subsets $\lambda^{-1}(\a)$, $\lambda^{-1}(\b)$, and $\lambda^{-1}(\c)$. Using this observation, we can give an alternative definition of the D-J class over $P_m$. A subset $I$ of $[m]$ is called \emph{non-consecutive} if $\{ p, p+1 \mod m\}$ are not contained in $I$ for any $p \in [m]$.
%

    \begin{definition}
        The set $\{ \mu_\a , \mu_\b, \mu_\c \}$ is a \emph{D-J class} over $P_m$ if it satisfies the following:
        \begin{enumerate}
          \item $\{ \mu_\a , \mu_\b, \mu_\c \}$ is a weak partition of $[m]$, that is, it is a partition of $[m]$ and one of $\mu_i$ could be empty.
          \item all of $\mu_\a , \mu_\b,$ and $\mu_\c$ are non-consecutive.
        \end{enumerate}
    \end{definition}


    \subsection{Edge of $D(P_m)$}

    The next step is for the edges of $D(P_m)$.

    \begin{definition}
        For a characteristic map $\lambda$ over $P_m$ and $p\in[m]$, the \emph{support} of $p$ for $\lambda$, denoted by $\supp_\lambda p$, is $\supp_\lambda p = \lambda^{-1}(\lambda(p))$. Since $\supp_\lambda p =\supp_{\lambda'} p$ whenever $\lambda $ and $\lambda'$ are D-J equivalent, the support $\supp_\lambda p$ is well defined for a D-J class $\lambda$.
    \end{definition}

    \begin{example}
        The characteristic map $\lambda$ over $P_7$ represented by
        \[
            \begin{pmatrix}
                1 & 0 & 1 & 0 & 1 & 1 & 0 \\
                0 & 1 & 1 & 1 & 0 & 1 & 1
            \end{pmatrix},
        \]
        or simply written as $\a\b\c\b\a\c\b$, is D-J equivalent to $\a\c\b\c\a\b\c$, and they correspond to the D-J class $\{\{1,5\},\{2,4,7\},\{3,6\}\}$ or simply $\{15,247,36\}$. The support of $4$ for $\lambda$ is $\supp_\lambda 4 = \{2,4,7\}$.
    \end{example}

    \begin{lemma}\label{lem:edge}
        Two D-J classes     $\lambda_1$ and $\lambda_2$ over $P_m$ are $p$-adjacent if and only if $\supp_{\lambda_1} p = \supp_{\lambda_2} p$.
    \end{lemma}
    \begin{proof}
        A proof is easily given by direct matrix calculation.
        Note that $\lambda_1$ and $\lambda_2$ are $p$-adjacent if and only if there is a characteristic matrix $\Lambda$ over $\wed_pP_m$ whose two projections are $\lambda_1$ and $\lambda_2$.
        As a version of Lemma~4.2 of \cite{CP_wedge_2} for the $\Z_2$-characteristic maps, it is known that for a characteristic map
            \[
                \lambda_1 = \begin{pmatrix}
                    \w_1 & \w_2 & \cdots & \w_m
                \end{pmatrix}_{n\times m}
            \]
        over a simplicial complex $K$ where $\w_i$ denotes the $i$th column vector of $\lambda_1$, the D-J class $\Lambda$ over $\wed_p K$ one of whose projection is $\lambda_1$ can be expressed by a matrix of the following form, called a \emph{standard form for an edge},
            \begin{equation}\label{eq:edge}
                \Lambda = \left(\begin{array}{cccccccc}
                    1    & \cdots & p-1      &  p_1 & p_2& p+1 & \cdots & m    \\ \hline
                    \w_1 & \cdots & \w_{p-1} & \w_p & 0  & \w_{p+1} & \cdots & \w_m \\
                    e_1  & \cdots & e_{p-1}  & 1   &  1 &  e_{p+1} & \cdots & e_m
                \end{array}\right)_{(n+1)\times (m+1)},
            \end{equation}
        for $e_i \in \Z_2$ for $i\ne p$. The numbers above the horizontal line are indicators for vertices of simplicial complexes.

        In our case, let us assume that $\lambda_1(p) = \lambda_2(p) = \binom{1}{0}$.
        For a suitable rearrangement of columns, $\Lambda$ can be written as
        \[
            \Lambda =
            \left(\begin{array}{ccccccccccc}
                1 & 0 & 1 & \cdots & 1  & 0 & \cdots & 0 & 1 & \cdots & 1 \\
                0 & 0 & 0 & \cdots & 0  & 1 & \cdots & 1 & 1 & \cdots & 1 \\
                1 & 1 & a_1 & \cdots & a_j & b_1 & \cdots &b_k  &c_1  & \cdots & c_\ell
            \end{array}\right),
        \]
        whose first two columns correspond to $\Lambda(p_1)$ and $\Lambda(p_2)$ respectively. Check that the projection $\lambda_1=\proj_{p_2}\Lambda$ is
        \[
            \begin{pmatrix}
                1 & 1 & \cdots & 1  & 0 & \cdots & 0 & 1 & \cdots & 1 \\
                0 & 0 & \cdots & 0  & 1 & \cdots & 1 & 1 & \cdots & 1
            \end{pmatrix}
        \]
        and $\lambda_2=\proj_{p_1}\Lambda$ is
        \[
            \begin{pmatrix}
                1 & 1+a_1 & \cdots & 1+a_j  & b_1 & \cdots &b_k  & 1+c_1 & \cdots & 1+c_\ell \\
                0 & 0     & \cdots & 0      & 1 & \cdots & 1 & 1 & \cdots & 1
            \end{pmatrix}.
        \]
        One can see that $a_1=\cdots=a_j = 0$ in order for $\lambda_2=\proj_{p_1}\Lambda$ to satisfy non-singularity condition, which is equivalent to that $\supp_{\lambda_1} p = \supp_{\lambda_2} p$.
    \end{proof}

    Recall that a characteristic map $\lambda$ over $P_m$ can be regarded as a finite circular sequence consisting of $\a,\b,$ and $\c$.
    If $S \subset [m]$ is a non-consecutive subset with $|S| \ge 2$, then $\lambda|_{[m]\setminus S}$ is divided into $|S|$ finite sequences called \emph{pieces of $\lambda$ determined by $S$}. If $S = \lambda^{-1}(\a)=\{p \mid \lambda(p) = \a\}$, then one obtains pieces consisting of $\b$ and $\c$, each of which will be called a \emph{$\b\c$-piece}. Likewise, \emph{$\c\a$-pieces} and \emph{$\a\b$-pieces} are defined.

    For a $\b\c$-piece $s$, the \emph{inversion} of $s$ is obtained from $s$ by exchanging $\b$ and $\c$.
    The following proposition completely describes the edges of $D(P_m)$.
    \begin{proposition}
        Let $\lambda_1$ and $\lambda_2$ be two characteristic maps over $P_m$ and $p \in [m]$. We also assume that $\lambda_1(p)=\lambda_2(p)=\a$. Then the D-J classes with respect to $\lambda_1$ and $\lambda_2$ are $p$-adjacent if and only if $\lambda_1$ becomes $\lambda_2$ after replacing a number of $\b\c$-pieces by their inversions.
    \end{proposition}
    \begin{proof}
        The proof is similar to Lemma~\ref{lem:edge}. Pick a $\b\c$-piece of $\lambda_1$ together $\a$'s on its ends. For example, consider a part of $\lambda_1$
        \[
            \begin{pmatrix}
                1 & 0 & 1 & 0 & 1 & \cdots & 1 & 0 & 1 \\
                0 & 1 & 1 & 1 & 1 & \cdots & 1 & 1 & 0
            \end{pmatrix}
        \]
        which corresponds to $\a\b\c\b\c\cdots \c \b\a$. According to the proof of Lemma~\ref{lem:edge},
        $\lambda_2$ should look like
        \[
            \begin{pmatrix}
                1 & 0+s_1 & 1+s_2 & 0+s_3   & \cdots & 1+s_{r-1} & 0+s_r & 1 \\
                0 & 1 & 1 & 1 &   \cdots & 1 & 1 & 0
            \end{pmatrix}
        \]
        at the same place, for $s_1,\dotsc,s_r \in \Z_2$. Since $\lambda_2$ satisfies non-singularity condition, we have $s_1 = \dotsb = s_r$. If $s_1=0$, then the corresponding $\b\c$-piece is fixed. If $s_1=1$, then the inversion is applied to the $\b\c$-piece. This works regardless of the shape of the $\b\c$-piece, completing the proof.
    \end{proof}

    \begin{example}
        Over $P_7$, assume that $\lambda = \a\b\c\b\a\c\b$ and $p=1$ or $p=5$. Then there are apparently three characteristic maps obtained by a number of inversions of $\b\c$-pieces: $\a\c\b\c\a\c\b$, $\a\b\c\b\a\b\c$, and $\a\c\b\c\a\b\c$. But since the whole exchanging of $\b$ and $\c$ does not change the D-J equivalence type, $\a\b\c\b\a\c\b \cong \a\c\b\c\a\b\c$ and $\a\c\b\c\a\c\b\cong \a\b\c\b\a\b\c$. In general, there are $2^{|\supp_\lambda p|-1}$ D-J classes $p$-adjacent to the D-J class with respect to $\lambda$.
    \end{example}

    Let $\lambda$ be a characteristic map over $P_m$ and let $s$, $t \in [m]$ be two vertices of $P_m$ such that $\lambda(s) = \lambda(t) = \a$. Then $s$ and $t$ divide $[m]\setminus \{s,t\}$ into two pieces $A$ and $B$. Then we  obtain a new characteristic map $\inv_{\{s,t\}}\lambda$ by replacing every $\b\c$-piece lying in $A$ by its inversion.  Although this definition depends on the choice of $A$, $\inv_{\{s,t\}}\lambda$ is well defined up to D-J equivalence provided $\lambda(s) = \lambda(t)$. More generally, let $S = \{s_1,\dotsc, s_{2\ell}\}\subset[m]$ be a non-consecutive subset of even cardinality and let $\lambda$ be a characteristic map over $P_m$ such that $\lambda(s_i) = \lambda(s_j)$ for $i,j=1,\dotsc, 2\ell$. Then we define
    \[
        \inv_S\lambda = \inv_{\{s_1,s_2\}}\inv_{\{s_3,s_4\}}\dotsb \inv_{\{s_{2\ell-1},s_{2\ell}\}}\lambda.
    \]
    One can easily see that the operation $\inv_S$ is well defined up to D-J equivalence and does not depend on the order of $s_1,\dotsc, s_{2\ell}$.
    \begin{remark}\label{rem:chessboard}
        Intuitively, $\inv_S$ can be understood as the following. One can regard $S$ as a set of $2\ell$ points on the circle. Then one paints each piece of the circle $S^1 \setminus S$ black or white like the chessboard, such that neighboring pieces have different colors. Then one performs inversion for every $\b\c$-piece in the black region. There are exactly two ways of such colorings, but $\inv_S$ is independent of coloring up to D-J equivalence.
    \end{remark}

    \begin{definition}
        For a D-J class $\lambda$ over $P_m$, an ordered pair $(p,S)$ is called an \emph{e-set} (compatible with $\lambda$) if it satisfies the following
        \begin{enumerate}
            \item $p \in [m]$,
            \item $S \subset [m]$ is of even cardinality, and
            \item 
            $S \subset \supp_\lambda p$.
        \end{enumerate}
    \end{definition}

    Note that $(p,\varnothing)$ is an e-set for arbitrary $\lambda$. We call $(p,\varnothing)$ an \emph{empty e-set}.
    We have a natural map $\epsilon$ which maps an edge $e=\{\lambda_1, \lambda_2, p\}$ of $D(P_m)$ to an e-set $\epsilon(e) = (p,S)$ in the following way: $p$ is the color of $e$, and $S$ is the unique set such that $\lambda_2  = \inv_S \lambda_1$.
    \begin{example}
        Let us consider the edge $e = \{\lambda_1,\lambda_2,p\}$ when $p=1$,
        \begin{align*}
            \lambda_1 &= \a \b\c \a \b \a \b \a \b \a \b,  \text{ and} \\
            \lambda_2 &= \a \b\c \a \c \a \c \a \b \a \c.
        \end{align*}
        Then $\supp_{\lambda_1}p = \{1,4,6,8,10\}$ and $\epsilon(e) = (1,\{1,4,8,10 \})$.
    \end{example}

    When $\lambda$ is a node incident to $e$ in $D(P_m)$ and $\epsilon(e) = (p,S)$, $\epsilon(e)$ is compatible with $\lambda$. Conversely, for any node $\lambda$ of $D(P_m)$ and its compatible e-set $(p,S)$, there is a unique edge $e$ incident to $\lambda$ such that $\epsilon(e) = (p,S)$.

    \subsection{Square of $D(P_m)$}
    The last step is to find realizable squares of $D(P_m)$. Recall that, if two edges $\{\lambda_1,\lambda_2,p\}$ and $\{\lambda_1,\lambda_3,q\}$ $(p\neq q)$ span a realizable square, then the fourth vertex is unique. That is, if both of the two squares
    \begin{equation*}
        \xymatrix{
                \lambda_{1} \ar@{-}[rr]^p \ar@{-}[dd]^q  & & \lambda_{2} \ar@{-}[dd]^q \\
                  &   &  \\
                \lambda_{3}\ar@{-}[rr]^p &   & \lambda_{4} }
        \hbox{ and }
        \xymatrix{
                \lambda_{1} \ar@{-}[rr]^p \ar@{-}[dd]^q  & & \lambda_{2} \ar@{-}[dd]^q \\
                  &   &  \\
                \lambda_{3}\ar@{-}[rr]^p &   & \lambda_{5} }
    \end{equation*}
    are realizable, then $\lambda_4 = \lambda_5$ as D-J classes.
    In order to find a realizable square, we firstly consider $\Lambda$ over $\wed_{p,q}K$ whose projections are given three characteristic maps $\lambda_1$, $\lambda_2$ and $\lambda_3$, namely, $\proj_{\{p_2,q_2\}}\Lambda = \lambda_1$, $\proj_{\{p_2,q_1\}}\Lambda = \lambda_2$ and $\proj_{\{p_1,q_2\}}\Lambda = \lambda_3$. Then, if $\Lambda$ exists, then it must be uniquely determined by $\lambda_1$, $\lambda_2$, and $\lambda_3$, as well as $\lambda_4=\proj_{\{p_1,q_1\}}\Lambda$.

    We have the matrix \eqref{eq:edge} for an edge of the diagram. For realizable squares, every D-J class over $\wed_{p,q} K$ for $p \ne q$ can be expressed by the following matrix called a \emph{standard form for a square},
        \begin{equation*}
            \left(\begin{array}{ccccccccccc}
                1    & \cdots  & p-1 & p_1 & p_2 & p+1 & \cdots & q_1 & q_2 & \cdots & m    \\ \hline
                \w_1 & \cdots  &    \w_{p-1} & \w_p & 0  & \w_{p+1} &  \cdots & \w_q & 0 & \cdots & \w_m \\
                e_1  & \cdots & e_{p-1} &     1   &  1 & e_{p+1}  & \cdots &  e_q & 0 & \cdots & e_m \\
                f_1  & \cdots & f_{p-1}  &   f_p   &  0 &  f_{p+1}  &  \cdots &  1 & 1 & \cdots & f_m
            \end{array}\right)_{(n+2)\times (m+2)},
        \end{equation*}
    by Proposition~4.3 of \cite{CP_wedge_2}. Note that its two projections with respect to $p_2$ and $q_2$ have the form of \eqref{eq:edge} when $v=q$ and $v=p$ respectively. We can regard them as the two edges $\{\lambda_1,\lambda_2,p\}$ and $\{\lambda_1,\lambda_3,q\}$.
    One can see that $\Lambda$ is a genuine characteristic matrix over $\wed_{p,q} P_m$ if and only if $\proj_{\{p_1,q_1\}}\Lambda$ satisfies the non-singularity condition as a characteristic map over $P_m$. In that case $\proj_{\{p_1,q_1\}}\Lambda = \lambda_4$ is the fourth node of the realizable square.

    By Lemma~\ref{lem:edge}, we know that $\supp_{\lambda_1}p = \supp_{\lambda_2}p$ and $\supp_{\lambda_1}q = \supp_{\lambda_3}q$. We have two cases depending on whether $\supp_{\lambda_1}p = \supp_{\lambda_1}q$ or not. Let us consider the first case when $\supp_{\lambda_1}p = \supp_{\lambda_1}q$.
    We may assume that $\lambda_1(p) = \lambda_1(q)= \a$. Let $\Lambda$ be a standard form
    over $\wed_{p,q}P_m$. Among vertices of $\wed_{p,q}P_m$, we pick $p_1,p_2,q_1,q_2$ and any other point $r$ as a representative and we write $\lambda_1(r) = \binom{v_1}{v_2}$. The submatrix of the matrix $\Lambda$ consisting of the columns corresponding to $p_1,p_2,q_1,q_2$, and $r$ would look like the following:
    \[
        \Lambda' = \begin{pmatrix}
            p_1 & p_2 & q_1 & q_2 & r \\ \hline
            1   & 0   & 1   & 0   & v_1 \\
            0   & 0   & 0   & 0   & v_2 \\
            1   & 1   & 0   & 0   & x \\
            0   & 0   & 1   & 1   & y
        \end{pmatrix},
    \]
    where $x,y \in \Z_2$. To compute $\lambda_2(r),\lambda_3(r),$ and $\lambda_4(r)$ from this matrix, we add the third and fourth row to the first one to obtain
    \[
        \begin{pmatrix}
            p_1 & p_2 & q_1 & q_2 & r \\ \hline
            0   & 1   & 0   & 1   & v_1+x+y \\
            0   & 0   & 0   & 0   & v_2 \\
            \underline{1}   & 1   & 0   & 0   & x \\
            0   & 0   & \underline{1}   & 1   & y
        \end{pmatrix}.
    \]
    Since the columns corresponding to $p_1$ and $q_1$ are coordinate vectors, one can obtain $\proj_{\{p_1,q_1\}}\Lambda'$ by deleting rows and columns containing the underlined entries:
    \[
        \begin{pmatrix}
            p_2 & q_2& r \\ \hline
            1   & 1  & v_1+x+y \\
            0   & 0  & v_2
        \end{pmatrix}.
    \]
    As consequence, we conclude that $\lambda_4(r) = \binom{v_1+x+y}{v_2}$ and similarly $\lambda_2(r) = \binom{v_1+x}{v_2}$ and $\lambda_3(r) = \binom{v_1+y}{v_2}$. When $v_2 = 0$, $v_1$ should be $1$ and one concludes $\lambda_1(r) = \binom{1}{0} = \a$ and $x=y=0$. When $v_2 = 1$, $\lambda_i(r) = \b$ or $\c$ for any $i=1,2,3,4$. In this case, observe that $x$ and $y$ indicate whether the $\b\c$-piece at the position of $r$ is inverted or not. Therefore we reach the following theorem.

    \begin{theorem}
        Consider the characteristic map $\lambda_1$ over $P_m$ of the form
        \[
            \lambda_1 = \a s_1 \a s_2 \a \cdots \a s_k,
        \]
        where $k = |\supp_\lambda 1|$ and $s_i$ are $\b\c$-pieces for $1\le i \le k$. Let $p$ and $q$ be distinct two vertices of $P_m$ such that $\lambda_1(p)=\lambda_1(q)=\a$. Then any two edges $\{\lambda_1,\lambda_2,p\}$ and $\{\lambda_1,\lambda_3,q\}$ span a realizable square with fourth node $\lambda_4$. Furthermore, one can assume that
        \[
            \lambda_2 = \a \phi_1(s_1) \a \phi_2(s_2) \a \cdots \a \phi_k(s_k)
        \]
        and
        \[
            \lambda_3 = \a \psi_1(s_1) \a \psi_2(s_2) \a \cdots \a \psi_k(s_k),
        \]
        where each of $\phi_i$ and $\psi_i$ is either the identity or the inversion, and the fourth node $\lambda_4$ is computed as follows:
        \[
            \lambda_4 = \a \psi_1(\phi_1(s_1)) \a \psi_2(\phi_2(s_2)) \a \cdots \a \psi_k(\phi_k(s_k)).
        \]
        \qed
    \end{theorem}

    \begin{example}
        Let $\lambda_1$, $\lambda_2$, and $\lambda_3$ be characteristic maps over $P_{11}$ represented by
        \begin{align*}
            \lambda_1 &= \a \b\c\b \a \c \a \b \a \c\b  \\
            \lambda_2 &= \a \b\c\b \a \b \a \b \a \b\c  \\
            \lambda_3 &= \a \c\b\c \a \b \a \b \a \b\c
        \end{align*}
        and pick distinct elements $p$ and $q$ in $\mu_\a=\{1,5,7,9\}$. Then the fourth node determined by the edges $\{\lambda_1,\lambda_2,p\}$ and $\{\lambda_1,\lambda_3,q\}$ is
        \[
            \lambda_4 = \a \c\b\c \a \c \a \b \a \c\b.
        \]
    \end{example}

    The remaining second case is when $\supp_{\lambda_1}p \ne \supp_{\lambda_1}q$. In this case, we can assume that $\lambda_i(p) = \a$ and $\lambda_i(q)=\b$ for all $i=1,\dotsc,4$. Just like before, let $\Lambda$ be a standard form for $\wed_{p,q}P_m$. We again pick $p_1,p_2,q_1,q_2$, and $r$ like before and we write $\lambda_1(r) = \binom{v_1}{v_2}$. The submatrix of the matrix $\Lambda$ consisting of the columns corresponding to $p_1,p_2,q_1,q_2$, and $r$ would look like the following:
    \[
        \begin{pmatrix}
            p_1 & p_2 & q_1 & q_2 & r \\ \hline
            1   & 0   & 0   & 0   & v_1 \\
            0   & 0   & 1   & 0   & v_2 \\
            1   & 1   & 0   & 0   & x \\
            0   & 0   & 1   & 1   & y
        \end{pmatrix},
    \]
    where $x,y \in \Z_2$. Again like before, we obtain
    \begin{equation}\label{eq:l2l3l4}
        \lambda_2(r) = \binom{v_1+x}{v_2},\, \lambda_3(r) = \binom{v_1}{v_2+y}\, \text{ and }\lambda_4(r) = \binom{v_1+x}{v_2+y}.
    \end{equation}
    One observes that if $\Lambda$ is non-singular then $xy = 0$ because if both of $x$ and $y$ were 1, then one of $\lambda_1(r),\,\lambda_2(r),\,\lambda_3(r),\,$ or $\lambda_4(r)$ should be zero, violating the non-singularity condition.

    Let $\mu$ and $\nu$ be two characteristic maps, not D-J classes, over $P_m$. We denote by $\mu\Delta\nu$ the set $\{v \in [m] \mid \mu(v)\ne\nu(v)\}$. For a subset $A \subseteq [m]$, we denote by $\overline{A}$ the set $\overline{A} = \{v+1,v,v-1 \mod m\mid v\in A \}\subseteq [m]$.

    \begin{lemma}\label{lem:squaretype2}
        Let $\lambda_1$, $\lambda_2$ and $\lambda_3$ be characteristic maps over $P_m$ such that $\lambda_i (p) =\a$ and $\lambda_i(q)=\b$ for $i=1,2,3$. Then the following are equivalent.
        \begin{enumerate}
            \item Two edges
        $\{\lambda_1,\lambda_2,p\}$ and $\{\lambda_1,\lambda_3,q\}$ span a realizable square.
            \item $\overline{\lambda_1\Delta\lambda_2} \cap {(\lambda_1\Delta\lambda_3)} =\varnothing$.
            \item $\overline{\lambda_1\Delta\lambda_2} \cap \overline{\lambda_1\Delta\lambda_3} =\varnothing$.
        \end{enumerate}
    \end{lemma}
    \begin{proof}
         Note that $r\in \lambda_1\Delta\lambda_2$ if and only if $x=1$ in \eqref{eq:l2l3l4}, and $r\in \lambda_1\Delta\lambda_3$ if and only if $y=1$ in \eqref{eq:l2l3l4}. In other words, $\lambda_1\Delta\lambda_2$ indicates the position of the inverted $\b\c$-pieces and $\lambda_1\Delta\lambda_3$ corresponds to the inverted $\c\a$-pieces. Since $xy=0$, one has $ {(\lambda_1\Delta\lambda_2)} \cap  {(\lambda_1\Delta\lambda_3)} =\varnothing$.

         Let us show the implication $(1) \Longrightarrow (2)$. Suppose that there is an element $r \in [m]$ such that $r \in \overline{\lambda_1\Delta\lambda_2} \setminus (\lambda_1\Delta\lambda_2)$ and $r \in \lambda_1\Delta\lambda_3$. Because $r$ is adjacent to an inverted $\b\c$-piece, $\lambda_1(r) = \a$. Without loss of generality, one can assume that $r+1 \in \lambda_1\Delta\lambda_2$. Then $r+1 \in \overline{\lambda_1\Delta\lambda_3} \cap {(\lambda_1\Delta\lambda_2)}$ and reminding that $ {(\lambda_1\Delta\lambda_2)} \cap  {(\lambda_1\Delta\lambda_3)} =\varnothing$, we obtain that $\lambda_1(r+1)= \b$ because $r+1 \in \overline{\lambda_1\Delta\lambda_3} \setminus (\lambda_1\Delta\lambda_3)$. Then \eqref{eq:l2l3l4} implies that $\lambda_4(r) = \lambda_4(r+1) = \c$, which contradicts to the non-singularity condition.

         To show the implication $(2) \Longrightarrow (3)$, consider parts of $\overline{\lambda_1\Delta\lambda_2}$ and $\overline{\lambda_1\Delta\lambda_3}$ where they could intersect. There are two possibilities of $\lambda_1$ restricted on $\overline{\lambda_1\Delta\lambda_2}$:
         \[
            \cdots\b\c\b\c\a \quad\text{ or }\quad \cdots \c\b\c\b\a.
         \]
         Similarly, there are two possibilities of $\lambda_1$ restricted on $\overline{\lambda_1\Delta\lambda_3}$:
         \[
            \b\a\c\a\c\cdots \quad\text{ or }\quad \b\c\a\c\a\cdots.
         \]
         In order to satisfy both $\overline{\lambda_1\Delta\lambda_2} \cap \overline{\lambda_1\Delta\lambda_3} \ne \varnothing$ and $ {(\lambda_1\Delta\lambda_2)} \cap  {(\lambda_1\Delta\lambda_3)} =\varnothing$, the only possibility is
         \[
            \cdots \c\b\c\b\c\a\c\a \cdots,
         \]
         where $\cdots \c\b\c\b$ lies on $\lambda_1\Delta\lambda_2$ and $\c\a\c\a\cdots$ is on $\lambda_1\Delta\lambda_3$. But it contradicts to the condition $(2)$.

         For the remaining part $(3) \Longrightarrow (1)$, one observes that \eqref{eq:l2l3l4} guarantees $\lambda_4$ is obtained from $\lambda_1$ by performing two inversion maps given by two edges $\{\lambda_1,\lambda_2,p\}$ and $\{\lambda_1,\lambda_3,q\}$. Since each inversion does not change the endpoints of $\overline{\lambda_1\Delta\lambda_2}$ and $\overline{\lambda_1\Delta\lambda_3}$, $\lambda_4$ is well defined and non-singular, proving that the square is realizable.
    \end{proof}
    The realizable square in Lemma~\ref{lem:squaretype2} is called a \emph{realizable square of type 2} if it is irreducible. A realizable square of $D(P_m)$ is called \emph{of type~1} if it is not of type 2.

    \begin{example}
        Put $p=1$, $q=2$ and let $\lambda_1$, $\lambda_2$, and $\lambda_3$ be characteristic maps over $P_{10}$ represented by
        \begin{align*}
            \lambda_1 &= \a\b \a \b\c\b \a \b \c\b  \\
            \lambda_2 &= \a\b \a \c\b\c \a \b \c\b  \\
            \lambda_3 &= \a\b \a \b\c\b \a \b \a\b.
        \end{align*}
        Then
        \[
            \lambda_4 = \a\b \a \c\b\c \a \b \a\b.
        \]
    \end{example}

    From now on, we study the realizable square in the language of e-sets, starting from the following definition. We fix a characteristic map $\lambda$ over $P_m$.
    \begin{definition}\label{def:type12related}
        When $p \ne q$, two e-sets $(p,S)$ and $(q,T)$ compatible with $\lambda$ are said to be \emph{type~1} or \emph{type~2 $\lambda$-related}, if edges $\{ \lambda, \inv_S \lambda, p\}$ and $\{ \lambda, \inv_T \lambda, q\}$ span a realizable square of type~1 or type~2, respectively.

        When $p=q$, $(p,S)$ and $(q,T)$ compatible with $\lambda$ are always said to be \emph{type~1 $\lambda$-related}.
    \end{definition}

    Now, let us find a combinatorial criterion for two given e-sets to be type~2 $\lambda$-related or not.
    Let $(p,S)$ be a non-empty e-set of $[m]$ and fix a vertex $r\in [m]\setminus S$. Recall Remark~\ref{rem:chessboard} to observe that the set $S$ divides $[m]\setminus S$ to two disjoint subsets $A$ and $B$ corresponding to ``black'' and ``white'' arcs respectively. We can assume that $r\in B$. Then we write $\overline{A} =: \Omega_r(S)$.

    \begin{proposition}\label{prop:type2}
        Assume that we are given two non-empty e-sets $(p,S)$ and $(q,T)$ compatible with $\lambda$. Then $(p,S)$ and $(q,T)$ are type 2 $\lambda$-related if and only if $\lambda(p)\neq \lambda(q)$ and $\Omega_q(S) \cap \Omega_p(T) = \varnothing$.
    \end{proposition}
    \begin{proof}
        To show the ``only if'' part, let us assume that $\epsilon(\{\lambda_1,\lambda_2,p\}) = (p,S)$ and $\epsilon(\{\lambda_1,\lambda_3,q\}) = (q,T)$. We can further assume that $\lambda_1(p) = \lambda_2(p) = \lambda_3(p)=\a$ and $\lambda_1(q) = \lambda_2(q) = \lambda_3(q) = \b$. Then observe that $\overline{\lambda_1\Delta\lambda_2} = \Omega_q(S)$ and $\overline{\lambda_1\Delta\lambda_3} = \Omega_p(T)$. Then, it immediately follows by Lemma~\ref{lem:squaretype2}.

        To show the ``if'' part, one picks any $\lambda$ such that
        \[\lambda(i)=\left\{
          \begin{array}{ll}
            \a, & \hbox{$i=p$ or $i\in S$;} \\
            \b, & \hbox{$i=q$ or $i\in T$}
          \end{array}
        \right. \]
        and the proof goes analogously.
    \end{proof}
    \begin{remark}\label{rem:type1}
        We remark that an empty e-set $(p, \varnothing)$ is always type~1 $\lambda$-related with arbitrary e-set $(q,T)$ compatible with $\lambda$. One can see that two non-empty e-sets $(p,S)$ and $(q,T)$ compatible with $\lambda$ are type~1 $\lambda$-related if and only if $\lambda(p) = \lambda(q)$.
    \end{remark}

\section{Classification of small covers over $P_m(J)$} \label{sec:classification_over_P_m(J)}
    We are now ready to describe the realizable puzzles over $P_m$. Let us fix $J = (j_1, \dotsc, j_m) \in \Z_+^m$ which is an $m$-tuple of positive integers. We denote by $DJ(m,J)$ the set of pairs $(\lambda,\cE)$ such that $\lambda$ is a node in $D(P_m)$ and $\cE$ is a finite sequence of e-sets of compatible with $\lambda$,
    \[
        (1,S^1_1),\dotsc,(1,S^1_{j_1-1}),(2,S^2_1),\dotsc,(2,S^2_{j_2-1}),\dotsc,(m,S^m_1),\dotsc,(m,S^m_{j_m-1}),
    \]
    such that the members of $\cE$ are pairwise $\lambda$-related.

    \begin{theorem}
        There is a bijection
        \[
            \{\text{small covers over $P_m(J)$ up to D-J equivalence}\} \longleftrightarrow DJ(m,J).
        \]
    \end{theorem}
    \begin{proof}
        By Theorem~\ref{thm:realizablepuzzle}, we will use the realizable puzzle instead of the D-J class over $P_m(J)$. We label nodes of $G(J)$ in the following way. Let $V(G(J))$ be the node set of $G(J)$. Recall that $G(J)$ is the 1-skeleton of the polytope
        \[
            \Delta^{j_1-1} \times \Delta^{j_2-1} \times \cdots \times \Delta^{j_m-1}.
        \]
        By labeling the vertices of $\Delta^{n-1}$ by $1,2,\dotsc,n$, we may identify $V(G(J)) = I(J)$ where
        \[
            I(J) = \{\balpha = (\alpha_1,\dotsc,\alpha_m )\mid  1 \le \alpha_p \le j_p\text{ for } p=1,\dotsc,m     \}.
        \]
        Let us use the notation $\1 := (1,\dotsc,1)$. A node $\balpha$ adjacent to $\1$ can be written as
        \[
            \balpha(p,\alpha_p) := (1,\dotsc,1,\alpha_p,1,\dotsc, 1),
        \]
        where the $p$th entry is $\alpha_p$, and the other entries are $1$.
        In order to construct a pair $(\lambda,\cE)$ from a realizable puzzle $\pi\colon G(J) \to D(P_m)$, we put $\lambda = \pi(\1)$ and $\cE$ the sequence consisting of $(\epsilon\circ\pi)(\{\1,\balpha(p,\alpha_p)\})$ when $1 \le p \le m$ and $1 \le \alpha_p \le j_p-1$. Then, $(\lambda, \cE)$ is indeed an element of $DJ(m, J)$. In summary, $(\lambda,\cE)$ is determined by $\pi(\1)$ and $\pi(\balpha)$ for all $\balpha$ adjacent to $\1$.

        The converse construction is the essential part of the proof. Our aim here is to construct a realizable puzzle $\pi$ using the data of $(\lambda,\cE)$. The basic philosophy is Proposition~\ref{prop:standardform}. We assign an e-set to each edge of $G(J)$ incident to $\1$ analogously to above argument. This is extended to the whole edges of $G(J)$ using the following rules.
        \begin{enumerate}
            \item Every parallel edge is assigned the same e-set.
            \item For a triangle and the corresponding e-sets $(p,S_i)$, $i=1,2,3$,
                \[
                    S_3 = (S_1 \cup S_2) \setminus (S_1 \cap S_2).
                \]
        \end{enumerate}
        This assignment is indeed well defined and the verification is very easy. Recall that $G(J)$ is an edge-colored graph. If $e$ is an edge of $G(J)$ whose color is $i$, then its assigned e-set is $(i,S)$ for some $S$.
        Now we define a pseudograph homomorphism $\pi\colon G(J) \to D(P_m)$ by the following rules.
        \begin{enumerate}
            \item $\pi(\1) = \lambda$.
            \item If $\balpha$ and $\bbeta$ are adjacent nodes in $G(J)$ and the assigned e-set to the edge $\{\balpha,\bbeta\}$ is $(p,S)$, then
                \[
                    \pi(\bbeta) = \inv_S\pi(\balpha).
                \]
        \end{enumerate}

        We must check that $\pi$ is well defined and is a realizable puzzle. The following diagram
        
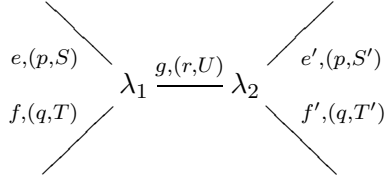
\[
            \xymatrix{
            \ar@{-}[rd]_{e,(p,S)} & & & \ar@{-}[ld]^{e',(p,S')}\\
            & \lambda_1 \ar@{-}[r]^{g,(r,U)} \ar@{-}[ld]_{f,(q,T)} &\lambda_2 \ar@{-}[rd]^{f',(q,T')} &\\
             & & & }
        \]
        indicates some nodes and edges of $G(J)$, whose two nodes map to $\lambda_1$ and $\lambda_2 = \inv_U \lambda_1$ by $\pi$ respectively. Here, $e,e', f,f',$ and $g$ are edges and the ordered pairs are assigned e-sets. We additionally assume that the $e$ and $e'$ are parallel and thus $S=S'$ if $p \ne r$, and the $f$ and $f'$ are parallel and $T=T'$ if $q \ne r$.

        In order to show both well-definedness and realizablity of $\pi$, we have to show the \emph{transitivity property} of e-sets: if $(p,S)$, $(q,T)$ and $(r,U)$ are mutually $\lambda_1$-related as e-sets compatible with $\lambda_1$, then $(p,S')$, $(q,T')$ and $(r,U)$ are mutually $\lambda_2$-related as e-sets compatible with $\lambda_2$.

        STEP 1. Let us firstly show that $(p,S')$, $(q,T')$ and $(r,U)$ are e-sets compatible with $\lambda_2$. Note that $\supp_{\lambda_1} r = \supp_{\lambda_2} r$ by the definition of $\inv_U$. Since $U \subset \supp_{\lambda_1} r$, we also have $U \subset \supp_{\lambda_2} r$. Therefore, $(r,U)$ is compatible with $\lambda_2$.

        Note that $S'$ is either $S$ or $(S \cup U) \setminus (S \cap U)$. If $\lambda_1(p) =\lambda_1(r)$, then
        $$
        \{p\} \cup S' \subset \{p\} \cup (S \cup U) \subset \supp_{\lambda_1}p = \supp_{\lambda_1}r = \supp_{\lambda_2}r.
        $$
        Since $p \in \supp_{\lambda_2} r$, we have $\supp_{\lambda_2} r = \supp_{\lambda_2} p$, and, thus
        \begin{equation}\label{eqn:lambda(p)=lambda(r)}
          \{p\} \cup S' \subset \supp_{\lambda_2} p.
        \end{equation}
        If $\lambda_1(p) \neq \lambda_1(r)$ (and thus $p \neq r$), then $(p,S)$ and $(r,U)$ must be type~2 $\lambda_1$-related. Hence, $\Omega_p(U) \cap \Omega_r(S) = \varnothing$. Note that since $p \not\in \Omega_p(U)$ and $S \subset \Omega_r(S)$, we have $(\{p \} \cup S) \cap \Omega_p(U) = \varnothing$. Assume that $\lambda_1(p)= \a$ and $\lambda_1(r)= \b$. Then, $\inv_U$ is the map exchanging $\a$ and $\c$ in $\Omega_p(U)$. Therefore,  $\{p \} \cup S \subset \supp_{\lambda_2} p$. Therefore,
        \begin{equation} \label{eqn:lambda(p)ne lambda(r)}
        \{p\} \cup S' = \{p \} \cup S \subset \left( \supp_{\lambda_2} p \cap \supp_{\lambda_1} p \right).
        \end{equation}
        In both cases, $(p,S')$ is compatible with $\lambda_2$.
        Similar arguments can be applied to $(q,T')$.

        STEP 2. Observe that $(r,U)$ and $(p,S')$ are $\lambda_2$-related since the three edges $e$, $e'$, and $g$ determine a realizable square when $p\ne r$, or a triangle when $p=r$. Similarly, $(r,U)$ and $(q,T')$ are $\lambda_2$-related.

        STEP 3. If $U = \varnothing$, then $\lambda_1 = \lambda_2$, $S'=S$ and $T'=T$. Therefore, since $(p,S)$ and $(q,T)$ are $\lambda_1$-related, $(p,S')$ and $(q,T')$ are $\lambda_2$-related. If $S = \varnothing$, then $S'=\varnothing$ or $S' = U$. In the second case, we must have $p=r$ and thus $(p,S') = (r,U)$ and we can use the result of STEP 2. Hence, in any case, $(p,S')$ and $(q,T')$ are $\lambda_2$-related. Similar argument can be applied to the case $T=\varnothing$. Hence, the claim holds when one of $S,T$ and $U$ is an empty set.

        From now on, let us assume that none of $S,T$ and $U$ is an empty set. We may use Proposition~\ref{prop:type2} and Remark~\ref{rem:type1} as criteria whether the e-sets are $\lambda_2$-related in the next Step.

        STEP 4. Let us show that $(p,S')$ and $(q,T')$ are $\lambda_2$-related. We divide this case into a few smaller cases:
       \begin{itemize}
            \item $\lambda_1(p) = \lambda_1(q) = \lambda_1(r)$: By \eqref{eqn:lambda(p)=lambda(r)}, we have
            $$ p \in \supp_{\lambda_2} p = \supp_{\lambda_2} r = \supp_{\lambda_2} q \ni q.$$  Therefore, $(p,S')$ and $(q,T')$ are type~1 $\lambda_2$-related.

            \item $\lambda_1(p) = \lambda_1(q) \ne \lambda_1(r)$ and $\Omega_p(U) = \Omega_q(U)$:
                Note that $\supp_{\lambda_1} p = \supp_{\lambda_1} q$, and observe that
                \begin{align*}
                 p  \in & \supp_{\lambda_2}p \setminus \Omega_p(U)
                   = \supp_{\lambda_1}p \setminus \Omega_p(U) \\
                  & = \supp_{\lambda_1}q \setminus \Omega_q(U)
                   = \supp_{\lambda_2}q \setminus \Omega_q(U) \ni q.
                 \end{align*} Therefore, $\lambda_2(p) =\lambda_2(q)$, and hence, $(p,S')$ and $(q,T')$ are type~1 $\lambda_2$-related.

            \item $\lambda_1(p) = \lambda_1(q) \ne \lambda_1(r)$ and $\Omega_p(U) \ne \Omega_q(U)$:
            Note that $\{U,[m]\setminus \Omega_p(U),[m]\setminus \Omega_q(U)\}$ is a partition of $[m]$. Also observe that
                \[
                    \{p\}\cup \Omega_r(S) \subset [m]\setminus \Omega_p(U)
                \]
                and
                \[
                    \{q\}\cup \Omega_r(T) \subset [m]\setminus \Omega_q(U).
                \]
                Thus $q\notin \Omega_r(S)$ and $\Omega_q(S) = \Omega_r(S)$. Likewise $\Omega_p(T) = \Omega_r(T)$, concluding $\Omega_q(S) \cap \Omega_p(T) = \varnothing$.
                Since $S'=S$ and $T'=T$,
                $(p,S')$ and $(q,T')$ are type~2 $\lambda_2$-related.

            \item $\lambda_1(p) \ne \lambda_1(q)$:
            If $\lambda_2(p)=\lambda_2(q)$, then $(p,S')$ and $(q,T')$ are type~1 $\lambda_2$-related.
            Now let us consider the case where $\lambda_2(p)\neq \lambda_2(q)$.
            In this case, $(p,S)$ and $(q,T)$ are type~2 $\lambda_1$-related.
                Hence, $\Omega_p(T) \cap \Omega_q(S) = \varnothing$.
                If $r\neq p$ and $r\neq q$, then $S=S'$ and $T=T'$. Therefore, $(p,S')$ and $(q,T')$ are type~2 $\lambda_2$-related.
                If $p=r$, then $S' =(S \cup U) \setminus (S \cap U)$, and $T'=T$ because $q\neq r$. Since $(q,T)$ and $(r,U)$ must be type~2 $\lambda_1$-related, we have $\Omega_p(T) \cap \Omega_q(U) = \Omega_r(T) \cap \Omega_q(U) =\varnothing$.
                To show $\Omega_p(T') \cap \Omega_q(S') = \varnothing$, we are enough to prove that
                \begin{equation}\label{eq:omegasu}
                    \Omega_q (S') \subset \Omega_q(S) \cup \Omega_q (U).
                \end{equation}
                To show \eqref{eq:omegasu}, pick any element $x \in [m] \setminus (S \cup U)$ reminding that the right hand side contains $S \cup U$. Then $[m] \setminus\{x,q\}$ is the disjoint union of two consecutive sets $A$ and $B$. Now $x$ is an element of $\Omega_q(S')$ if and only if $|S' \cap A|$ is odd. Thus one of $|S \cap A|$ or $|U \cap A|$ must be odd, and $x \in \Omega_q(S) \cup \Omega_q (U)$.  In conclusion, $(p,S')$ and $(q,T')$ are type~2 $\lambda_2$-related.
        \end{itemize}
        The above Steps 1--4 prove our claim.

        Let us finish our goal using the above claim we just have shown. Pick a node $\balpha = (\alpha_1,\dotsc,\alpha_m )\in I(J)$ of $G(J)$ and consider the minimal path from $\1$ to $\balpha$ given by the sequence
        \[
            \1=(1,\dotsc,1),\,(\alpha_1,1,\dotsc,1),\,(\alpha_1,\alpha_2,1,\dotsc,1),\dotsc,(\alpha_1,\dotsc,\alpha_m )=\balpha.
        \]
        Then $\pi(\balpha)$ is given by a chain of $\inv$ maps
        \[
            \pi(\balpha) = \inv_{S_m}  \inv_{S_{m-1}}  \dotsb \inv_{S_2}  \inv_{S_1} \lambda
        \]
        for the e-sets $(1,S_1),\dotsc,(m,S_m)$ corresponding to the above sequence. To check that this is well defined, we start from the node $\1$. The edge connecting $\1$ and $(\alpha_1,1,\dotsc,1)$ and that connecting $\1$ and $(1,\alpha_2,1,\dotsc,1)$ spans a realizable square since the e-sets $(1,S_1)$ and $(2,S_2)$ are $\lambda$-related. Therefore $\inv_{S_2}  \inv_{S_1} \lambda$ is well defined. Next, we move on the node $(\alpha_1,1,\dotsc,1)$. Then the claim shows that $(2,S_2)$ and $(3,S_3)$ are $\pi(\alpha_1,1,\dotsc,1)$-related and thus $\inv_{S_3}\inv_{S_2}  \inv_{S_1} \lambda$ is well defined. The inductive application of this procedure proves that $\pi$ is well defined for every node of $G(J)$. Similarly, our claim shows that every subsquare is realizable, completing the proof.
    \end{proof}

    For a D-J class $\lambda$ over $P_m$, put
    $$
        E(\lambda,J) = \{ \cE \mid (\lambda, \cE) \in DJ(m,J) \}.
    $$

    \begin{corollary}\label{cor:elj}
      The number of small covers over $P_m(J)$ is equal to the sum of $|E(\lambda,J)|$ for all D-J classes $\lambda$ over $P_m$.
    \end{corollary}

    Some examples and calculations for Corollary~\ref{cor:elj} will be given in Section~\ref{sec:summary}.

\section{Real toric manifolds over $P_m(J)$} \label{sec:real_toric_manifold_over_P_m(J)}
    The objective of this section is to specify every real toric manifold over the wedged polygon in terms of realizable puzzles. In \cite{CP_CP_variety}, the authors have described all smooth complete toric varieties over wedged polygons $P_m(J)$ up to D-J equivalence. By taking the mod 2 reduction, one can obtain the classification of real toric manifolds over wedged polygons up to D-J equivalence. Before that we need some preperation. We identify $[m]$ with the vertex set of $P_m$ as before. Let us start with the real toric manifolds over $P_m$.
    \begin{lemma}\label{lem:realtoricoverpolygon}
         Let $\lambda\colon [m] \to \Z_2^2$ be a $\Z_2$-characteristic map over $P_m$. The small cover given by $\lambda$, written as $M(\lambda)$, is also a real toric manifold if and only if
         \begin{enumerate}
            \item $\lambda$ is D-J equivalent to $\begin{pmatrix} \a & \b & \a  & \b \end{pmatrix}$, or
            \item the image of $\lambda$ has exactly three elements.
         \end{enumerate}
          In other words, $M(\lambda)$ is a real toric manifold unless $\lambda$ is D-J equivalent to
          \[
                \Lambda_{2k} := \begin{pmatrix} \a & \b & \a  & \b & \cdots & \a & \b \end{pmatrix}_{2 \times 2k},
          \]
          where $k\ge 3$.
    \end{lemma}
    \begin{proof}
        First of all, recall the classical fact that every complete smooth toric surface is either $\CP^2$ or a consecutive equivariant blow-up of a Hirzebruch surface. Then their mod 2 reductions cannot be D-J equivalent to $\Lambda_{2k}$ for $k\ge 3$. More precisely, we can define a \emph{blow-up} of a $\Z_2$-characteristic map $\begin{pmatrix} \v_1 & \cdots  & \v_m \end{pmatrix}$ over $P_m$ to be
        \[
            \begin{pmatrix} \v_1 & \cdots & \v_i &  \v_i + \v_{i+1} & \v_{i+1}  & \cdots & \v_m \end{pmatrix}
        \]
        for $i=1,\dotsc, m-1$, or
        \[
            \begin{pmatrix} \v_1 & \cdots & \v_m & \v_m+\v_1 \end{pmatrix}
        \]
        for $i=m$, which is a $\Z_2$-characteristic map over $P_{m+1}$. Observe that $\Lambda_{2k}$ is not a blow-up of other $\Z_2$-characteristic map, and $\Lambda_{2k}$ gives a real toric manifold if and only if $k=2$.

        Conversely, we are going to show that every $\lambda\colon [m] \to \Z_2^2$ with three values gives a real toric manifold. One can assume that $m \ge 5$ and we use an induction on $m$.  Note that there is an inverse operation of the blow-up for $\lambda$, called a \emph{blow-down}, if there is a vertex $p$ of $P_m$ such that $\lambda(p)$, $\lambda(p+1)$, and $\lambda(p+2)$ are all distinct, when $p$, $p+1$, and $p+2$ are considered modulo $m$. When $\lambda$ is viewed as a circular sequence, it should have a subsequence $\a\c\b$ up to D-J equivalence, i,e, $\lambda(p)=\a$, $\lambda(p+1)=\c$, and $\lambda(p+2) = \b$. If $|\lambda^{-1}(\c)| \ge 2$, then the blow-down obtained by deleting $\c$ in the subsequence also has three values. If $p+1$ is the only one with the value $\c$, then we have $\lambda(p-1) = \b$ for $p-1 \ne p+2 \mod m$ and we have a subsequence $\b\a\c\b$. Therefore there is an blow-down obtained by deleting $\a$. Since $|\lambda^{-1}(\a)| \ge 2$, the blow-down still has three values. By induction, $\lambda$ is obtained by consecutive blow-ups from a $\Z_2$-characteristic map over $P_4$ and the proof is complete.
    \end{proof}

    The next step is for real toric manifolds over $\wed_pP_m$. The following is from \cite{CP_CP_variety}.

    \begin{proposition}\cite[Proposition~3.1]{CP_CP_variety}\label{prop:shift} Let $\Sigma_1$ and $\Sigma_2$ be two complete non-singular fans with $m$ rays in $\R^2$ and the matrix
        \[
            \lambda = \begin{pmatrix}
                1 & 0   & x_3   & x_4 & \cdots & x_m \\
                0 & 1   & y_3   & y_4 &  \cdots & y_m
            \end{pmatrix}
        \]
        a characteristic matrix for $\Sigma_1$. Suppose that the two fans are 1-adjacent in the diagram $D^{\text{toric}}(P_m)$ for toric manifolds. Then either of the following holds:
        \begin{enumerate}
            \item two fans are the same.
            \item two fans share a ray generated by $\binom{-1}0$.
        \end{enumerate}
        In the second case, a characteristic matrix for $\Sigma_2$ can be written as the following:
        \[
            \zeta_e := \begin{pmatrix}
                1 & 0 & x_3 & \cdots & x_{\ell-1} & x_\ell = -1 & x_{\ell+1}+e  & x_{\ell+2} - y_{\ell+2}e & \cdots & x_m - y_m e \\
                0 & 1 & y_3 & \cdots & y_{\ell-1} & y_\ell = 0  & y_{\ell+1}=-1 & y_{\ell+2}               & \cdots & y_m
            \end{pmatrix}
        \]
        for some $e$ and $\ell$. Conversely, for every $e\in\Z$, $\zeta_e$ is 1-adjacent to $\lambda$ whenever $\lambda$ contains a ray generated by $\binom{-1}0$.
    \end{proposition}

    An edge $\{\lambda_1,\lambda_2,p\}$ in $D(P_m)$ is said to be \emph{real toric} if either $\lambda_1 = \lambda_2$ or the following holds:
    \begin{enumerate}
        \item there exists $q\in [m]$ so that $\lambda_2 = \inv_{\{p,q\}} \lambda_1$, and
        \item each piece of $\lambda_1$ determined by $\{p,q\}$ does not contain $\lambda_1(p)$ or contains all three values.
    \end{enumerate}
    {An edge satisfying above $(1)$ can be represented by the set $\{\lambda_1,(p,q)\}$, or a circular sequence consisting of $\a$, $\b$, and $\c$ whose two points at $p$ and $q$ are marked with different marks $()$ and $\{\}$ respectively. For example, consider the edge $\{\lambda_1,\lambda_2,1\}$ where
    \begin{align*}
        \lambda_1 &= \a \b\c\b \a \c \a \b  \text{ and} \\
        \lambda_2 &= \a \b\c\b \a \b \a \c.
    \end{align*}
    It can be represented by $\{\a \b\c\b \a \c \a \b,(1,5)\}$ or $(\a) \b\c\b \{\a\} \c \a \b$. It is real toric because $\b\c\b$ does not contain $\a$ and $\c\a\b$ contains all three values.
    On the other hand, the edge $ (\a) \b\c \{\a\} \b \a \b $  is not real toric  since $\b\a\b$ contains $\a$ and has only two values.
    }
    \begin{lemma}\label{lem:realtoricedge}
         An edge in $D(P_m)$ gives a real toric manifold over $\wed_pP_m$ for some $p\in [m]$ if and only if it is real toric.
    \end{lemma}
    \begin{proof}
        Let us denote by $\lambda_\R$ the mod 2 reduction of a characteristic map $\lambda \colon [m] \to \Z^n$.
        Let us deal with ``If'' part first. That is, we are given an edge satisfying $(1)$ and $(2)$. 
        We start from the circular sequence $(\a)\b\{\a\}\b$  or $(\a)\b\{\a\}\c$ and apply consecutive blow-ups to finally obtain any edge satisfying $(1)$ and (2). Observe that every edge obtained in this way corresponds to a real toric manifold over $\wed_pP_m$ by Proposition~\ref{prop:shift}. Or, equivalently, one starts from a circular sequence with two different marks satisfying (1) and (2), and performs blow-downs conserving two marks to end with $(\a)\b\{\a\}\b$  or $(\a)\b\{\a\}\c$. First, let us consider the case that each piece of $\lambda_1$ determined by $\{p,q\}$ does not contain $\a$. Then each piece consists of $\b$ and $\c$ and for example it will look like
        \[
            (\a) \b \c \b \dotsb \c \b \c \{\a\},
        \]
        together with the two marked $\a$'s. Then one can repeat blow-downs until the piece has length 1, obtaining $(\a)\b\{\a\}$ or $(\a)\c\{\a\}$. When a piece contains $\lambda_1(t)=\a$, watch for its two neighbors at $t-1$ and $t+1 \mod m$. If they are different to each other, then we can remove $\a$ at $t$ by blow-down. If not, then without loss of generality $\lambda_1(t-1) = \lambda_1(t+1) = \b$. Then one searches for $\c$ at $t-2$ and $t+2$ $\mod m$ in the piece and at $t-3$ and $t+3$ and so on. Since the piece contains $\c$ by condition (2), one can eventually find $\c$ in the piece. For example, the following piece
        \[
           \b \underline{\a} \b \a \b \dotsb \a \b \c,
        \]
        where the underlined $\a$ is at the position $t$, can be reduced to
        \[
            \b \underline{\a} \c
        \]
        by consecutive blow-downs. Finally we can remove $\a$ on the underline.

        For ``only if'' part, the condition (1) is a direct consequence of Proposition~\ref{prop:shift}. Suppose that the condition (2) does not hold. Then we have a non-singular fan $\Sigma$ in $\R^2$ spanning the upper half-plane whose mod 2 reduction has the form $\a \b \a \b \dotsb \a \b \a$ up to D-J equivalence. Then by reflection of $\Sigma$ across the $x$-axis, we obtain a complete non-singular fan whose mod 2 reduction has only two values $\a$ and $\b$. This is a contradiction with Lemma~\ref{lem:realtoricoverpolygon} and the proof is done.
    \end{proof}

    Let $J=(j_1,\dotsc,j_m)\in \Z_+^m$ be an $m$-tuple of positive integer.
    \begin{theorem}
        A realizable puzzle $\pi \colon G(J) \to D(P_m)$ corresponds to a real toric manifold if and only if either every node of $G(J)$ maps to the same node of $D(P_m)$, or there is a quadruple $(p,q,\lambda_1,\lambda_2)$ such that
        \begin{enumerate}
            \item $p,q \in [m]$ and $\lambda_2 = \inv_{\{p,q\}} \lambda_1$,
            \item the edges $\{\lambda_1,\lambda_2,p\}$ and $\{\lambda_1,\lambda_2,q\}$ are real toric, and
            \item every nontrivial edge of $G(J)$ maps to either the edge $\{\lambda_1,\lambda_2,p\}$ or $\{\lambda_1,\lambda_2,q\}$.
        \end{enumerate}
        In the latter case, every irreducible realizable square has the form
        \begin{equation*}
        \xymatrix{
                \lambda_{1} \ar@{-}[rr]^p \ar@{-}[dd]^q  & & \lambda_{2} \ar@{-}[dd]^q \\
                  &   &  \\
                \lambda_{2}\ar@{-}[rr]^p &   & \lambda_{1} }.
        \end{equation*}
    \end{theorem}
    \begin{proof}
        The proof is an immediate consequence of Proposition~3.2 of \cite{CP_CP_variety} and Lemma~\ref{lem:realtoricedge}.
    \end{proof}

\section{Summary}\label{sec:summary}
    This self-contained section is the summary for this article. We describe our classification in terms of combinatorial language and provide several counting examples. We shall introduce again all notions required to understand the main results.

    Let us fix a positive integer $m \geq 3$. For a finite set $S$, a \emph{weak partition} $\{S_1, \ldots, S_k\}$ of $S$ is a partition of $S$ such that some of $S_i$ could be empty. A subset $I$ of $[m]=\{1,2,\ldots,m\}$ is said to be \emph{non-consecutive} if $\{ p, p+1 \mod m\}$ are not contained in $I$ for any $p \in [m]$.

    We consider a weak partition $\lambda = \{ \mu_\a , \mu_\b, \mu_\c \}$ of $[m]$ such that all of $\mu_\a , \mu_\b,$ and $\mu_\c$ are non-consecutive. We call $\lambda$ a \emph{D-J class over $P_m$}.
    Two elements $p$ and $q$ of $[m]$ are said to be \emph{$\lambda$-equivalent} if $\{p,q\} \subset \mu_i$ for some $i$.
    \begin{remark}
        It is convenient to describe a D-J class over $P_m$ using a word consisting of alphabets $\a$, $\b$, and $\c$. For example, the D-J class $\{\{1,3\},\{2\},\{4\}\}$ can be written as $\a\b\a\c$ or $\b\c\b\a$; the representation is unique up to permutation of $\a$, $\b$, and $\c$.
    \end{remark}

    \begin{definition}
        For a D-J class $\lambda = \{ \mu_\a, \mu_\b, \mu_\c\}$ over $P_m$, an ordered pair $(p,S)$ is called an \emph{e-set} (compatible with $\lambda$) if it satisfies the following
        \begin{enumerate}
            \item $p \in [m]$,
            \item $S \subset [m]$ is of even cardinality, and
            \item all elements of $\{p\}\cup S$ are mutually $\lambda$-equivalent, that is, $\{p\}\cup S \subset \mu_i$ for some $i$.
        \end{enumerate}

        In particular $(p, \varnothing)$ is always an e-set for any $\lambda$. 
    \end{definition}

    Let $(p,S)$ be an e-set of $[m]$ and fix a vertex $r\in [m]\setminus S$.
    Note that $S$ divides $[m]\setminus S$ to two disjoint subsets $A$ and $B$ corresponding to ``black'' and ``white'' arcs respectively. We can assume that $r\in B$. Then we write $\Omega_r(S):=A \cup S$.

    \begin{definition}
        Let $(p,S)$ and $(q,T)$ be two e-sets compatible with $\lambda$.
        \begin{enumerate}
            \item Two e-sets $(p,S)$ and $(q,T)$ of $[m]$ are type~1 $\lambda$-related if and only if  either one of $S$ and $T$ is empty, or $p$ and $q$ are $\lambda$-equivalent.
            \item Two e-sets $(p,S)$ and $(q,T)$ of $[m]$ are type~2 $\lambda$-related if and only if $S$ and $T$ are nonempty, $p$ and $q$ are not $\lambda$-equivalent, and $\Omega_q(S) \cap \Omega_p(T) = \varnothing$.
        \end{enumerate}
        In addition, $(p,S)$ and $(q,T)$ are said to be \emph{$\lambda$-related} if they are either type~1 $\lambda$-related or type~2 $\lambda$-related.
    \end{definition}

    For a D-J class $\lambda$ over $P_m$ and a positive integer $m$-tuple $J=(j_1, \ldots, j_m)$, we denote by $E(\lambda,J)$ the set of finite sequences $\cE$ of e-sets of $[m]$ compatible with $\lambda$
    \[
        (1,S^1_1),\dotsc,(1,S^1_{j_1-1}),(2,S^2_1),\dotsc,(2,S^2_{j_2-1}),\dotsc,(m,S^m_1),\dotsc,(m,S^m_{j_m-1})
    \]
    such that the members of $\cE$ are mutually $\lambda$-related.

    \begin{theorem}
        The number of small covers over $P_m(J)$ up to D-J equivalence is
        \[
            \sum_\lambda{|E(\lambda,J)|},
        \]
        where $\lambda$ runs through all D-J classes over $P_m$.%
    \end{theorem}

For a given $\lambda$, it should be interesting to consider a subset $\tilde{E}(\lambda, J)$ of $E(\lambda,J)$, which is the set of $\cE$ such that its e-sets are mutually type~1 $\lambda$-related.

\begin{proposition}\label{prop:tilde_E_formula}
    Let $\lambda = \{ \mu_\a, \mu_\b, \mu_\c \}$ be a D-J class over $P_m$, and $J=(j_1, \ldots, j_m) \in \Z_+^m$. Then,
$$
    |\tilde{E}(\lambda,J)|=(2^{|\mu_a|-1})^{w_\a} + (2^{|\mu_b|-1})^{w_\b} + (2^{|\mu_c|-1})^{w_\c} - 2,
$$where $w_i = \left\{
                 \begin{array}{ll}
                   0, & \hbox{ if $\mu_i=\varnothing$;} \\
                   \sum_{ k \in \mu_i } (j_k-1), & \hbox{otherwise.}
                 \end{array}
               \right.
$.
\end{proposition}
\begin{proof}
Note that the number of even-subsets of $\mu_i$ is $2^{|\mu_i|-1}$.
For each $p \in \mu_i$, we have to choose $(p,S)$ such that $S$ is an even-subset of $\mu_i$.
One remarks that if $p \in \mu_i$ and $(p,S) \in \cE$ with a non-empty set $S$, then $(q,T) \in \cE$ implies that either $q \in \mu_i$ or $T =\varnothing$. Therefore, there are $(2^{|\mu_i|-1})^{w_i}$ cases. Together with the case that every $S$ is an empty set, the proposition is proved.
\end{proof}

\begin{example}[Example~2.9 of \cite{Choi2008}]\label{example:m=4}
    Assume that $m=4$ and $J=(j_1,j_2,j_3,j_4)$. There are only three D-J classes $\lambda_1 = \a\b\a\b$, $\lambda_2=\a\b\a\c$, and $\lambda_3=\a\b\c\b$. For each $\lambda_i$, $\lambda_i$-related $(p, S)$ and $(q,T)$ are always of type~1, that is, $E(\lambda_i,J) = \tilde{E}(\lambda_i,J)$. By Proposition~\ref{prop:tilde_E_formula}, we have
    \begin{align*}
    |\tilde{E}(\lambda_1,J)| &= 2^{j_1 + j_3-2} + 2^{j_2+j_4 -2} -1, \\
    |\tilde{E}(\lambda_2,J)| &= 2^{j_1 + j_3-2}, \text{ and } \\
    |\tilde{E}(\lambda_3,J)| &= 2^{j_2 + j_4-2}.
    \end{align*}
    Thus, the number of D-J classes over $P_4(J)$  is
    $$
        2^{j_1 + j_3-1} + 2^{j_2 + j_4-1} -1.
    $$
\end{example}

\begin{example}[Theorem~8.3 of \cite{Choi-Park2016}]
    Assume that $m=5$ and $J=(j_1,\ldots,j_5)$. Then, there are five D-J classes over $P_5$. Put $$\lambda_i \equiv \{ \{i\}, \{i+1, i+3\}, \{i+2, i+4\} \} \text{ mod $5$}.$$
    Similarly to Example~\ref{example:m=4}, for each $\lambda_i$, one can see that $E(\lambda_i, J) = \tilde{E}(\lambda_i,J)$, and we have
    $|\tilde{E}(\lambda_i,J)|=2^{j_{i+1} + j_{i+3}-2} + 2^{j_{i+2} + j_{i+4}-2}-1$ where all indices are given by modulo $5$. Therefore, the number of D-J classes over $P_5(J)$ is
    $$
    2^{j_1 + j_3-1} + 2^{j_2 + j_4-1} + 2^{j_3 + j_5-1} + 2^{j_4 + j_1-1} + 2^{j_5 + j_2-1} -5.
    $$
\end{example}

\begin{example}
    Assume that $m=6$ and $J=(j_1,\ldots,j_6)$. There are $11$ D-J classes over $P_6$. In order to compute the number of D-J classes over $P_6(J)$, since we can compute $|\tilde{E}(\lambda,J)|$ easily due to Proposition~\ref{prop:tilde_E_formula}, it is enough to consider $E(\lambda,J) \setminus \tilde{E}(\lambda,J)$.
We note that there are only a few possible pairs of $(p,S)$ which can be type 2 related:
$$
    (i+2 \pm1 ,\{i+1,i+3\}), (i+5 \pm1,\{i+4,i+6\}) \quad \text{for $i=1,2,3$}.
$$
Here is the list of all type~2 $\lambda$-related pairs for each $\lambda$.

\begin{tabular}{|c|c|}
  \hline
  Type~2 $\lambda$-related pair & corresponding $\lambda$ \\ \hline
  $(2 \pm1 ,\{1,3\}), (5 \pm1,\{4,6\})$ & $\a\b\a\b\a\b$, $\a\b\a\b\c\b$, $\a\c\a\b\a\b$, $\a\c\a\b\c\b$ \\   \hline
  $(4 \pm1 ,\{3,5\}), (1 \pm1,\{6,2\})$ & $\a\b\a\b\a\b$, $\a\b\a\c\a\b$, $\c\b\a\b\a\b$, $\c\b\a\c\a\b$ \\   \hline
  $(6 \pm1 ,\{5,1\}), (3 \pm1,\{2,4\})$ & $\a\b\a\b\a\b$, $\a\b\a\b\a\c$, $\a\b\c\b\a\b$, $\a\b\c\b\a\c$ \\
  \hline
\end{tabular}

Hence, one can see that the sum of $|E(\lambda,J) \setminus \tilde{E}(\lambda,J)|$ for all D-J classes $\lambda$ over $P_6$ is
\begin{align*}
   4 \times ((2^{j_1+j_3-2}-1)(2^{j_4+j_6-2}-1) + (2^{j_3+j_5-2}-1)(2^{j_6+j_2-2}-1)  \\+ (2^{j_5+j_1-2}-1)(2^{j_2+j_4-2}-1)).
 \end{align*}

\end{example}

Real toric manifolds over $P_m(J)$ are small covers and there is an analogue of $E(\lambda,J)$ for real toric manifolds which is a subset of $E(\lambda,J)$.
Let us fix a D-J class $\lambda = \{\mu_\a, \mu_\b,\mu_\c\}$ over $P_m$ and pick two distinct points $p,\,q \in [m]$ which are $\lambda$-equivalent. Then we can assume that $p,q \in \mu_\a$. The points $p$ and $q$ divide $[m] \setminus \{p,q\}$ to two disjoint sets $A_1$ and $A_2$ in the way explained above. In this settings, we say $p$ and $q$ \emph{satisfies property RT with $\lambda$} if $\mu_\a \cap A_i = \varnothing$ or $\mu_\a \cap A_i,\, \mu_\b\cap A_i$, and $\mu_\c\cap A_i$ are all nonempty, for each $i=1,2$.

\begin{definition}
    Let $\lambda$ be a D-J class over $P_m$ for $m\ge 4$. We denote by $E_{\text{RT}}(\lambda,J)$ the set of $\cE \in E(\lambda,J)$ satisfying the following; for such a sequence $\cE$, one can take two elements $p,q \in [m]$ satisfying property RT with $\lambda$ such that every entry of $\cE$ is $(i,\varnothing)$, $(p,\{p,q\})$, or $(q,\{p,q\})$  for $1\le i \le m$. When $\lambda = \{\{1\},\{2\},\{3\}\}$ is the unique D-J class over $P_3$, we define $E_{\text{RT}}(\lambda,J)$ to be the singleton whose unique element is $\cE$ so that every entry of $\cE$ is $(i,\varnothing)$ for $1\le i \le m$.
\end{definition}

\begin{remark}
 For $k \ge 2$, we have a family of D-J classes $\Lambda_{2k} = \{\{1,3,\dotsc,2k-1\},\{2,4,\dotsc,2k\},\varnothing \}$.
    When $\lambda$ is a D-J class over $m \ge 3$, $E_{\text{RT}}(\lambda,J) = \varnothing$ if and only if $\lambda = \Lambda_{2k}$ for $k\ge 3$. In fact, there is no pair $p$ and $q$ satisfying property RT with $\Lambda_{2k}$ for $k\ge 3$.
\end{remark}

    \begin{theorem}
        The number of real toric manifolds over $P_m(J)$ is
        \[
            \sum_\lambda{|E_{\text{RT}}(\lambda,J)|},
        \]
        where $\lambda$ runs through all D-J classes over $P_m$.
    \end{theorem}

\begin{remark}
    For any D-J class $\lambda$ over $P_m$ and $J \in \Z_+^m$, we have the inclusions
    \[
           E_{\text{RT}}(\lambda,J) \subset \tilde{E}(\lambda, J) \subset E(\lambda, J).
    \]
    When $m \le 5$, the two inclusions are actually equalities (see Chapter~8 of \cite{Choi-Park2016}). But in general, almost all inclusions are strict. Indeed, they are strict if $m\ge 6$ and $J \ne (1,\dotsc,1)$.
\end{remark}

\bigskip

\bibliographystyle{amsplain}

%
%
%
%
%
%
%
%
%
%
%
%
%


\end{document}